\documentclass[12pt,a4paper,leqno]{article}

\usepackage[american]{babel}
\usepackage{amsmath,amsthm}
\usepackage{latexsym}
\usepackage{amssymb}
\usepackage[shortlabels]{enumitem}
\usepackage{tikz}
\usepackage{tikz-cd}

\renewcommand{\phi}{\varphi}
\newcommand{\RR}{\mathbb{R}}
\newcommand{\ZZ}{\mathbb{Z}}
\newcommand{\FF}{\mathbb{F}}

\newcommand{\CC}{\mathbb{C}}
\newcommand{\Sphere}{\mathbb{S}}
\newcommand{\KK}{\mathbb{K}}
\newcommand{\bfone}{\boldsymbol{1}}
\newcommand{\id}{\mathrm{id}}

\newcommand{\srg}{\mathrm{srg}}
\newcommand{\tr}{\mathrm{tr}}
\newcommand{\J}{\mathbf{J}}
\newcommand{\K}{\mathbf{K}}
\newcommand{\sfR}{\mathsf{R}}
\newcommand{\eps}{\varepsilon}

\newcommand{\cQ}{\mathcal{Q}}

\theoremstyle{plain}
\newtheorem{Thm}[subsection]{Theorem}

\newtheorem{Conj}[subsection]{Conjecture}
\newtheorem{Lem}[subsection]{Lemma}

\theoremstyle{definition}

\newtheorem{Facts}[subsection]{Facts}

\newtheorem{Ex}[subsection]{Example}
\newtheorem{Rem}[subsection]{Remark}
\newtheorem{Def}[subsection]{Definition}

\newtheorem{Num}[subsection]{}

\allowdisplaybreaks

\title{\bf Weyl Tensors, Strongly Regular Graphs, Multiplicative Characters, and a Quadratic Matrix Equation}
\author{Christopher Deninger, Theo Grundh\"ofer and Linus Kramer
\thanks{CD and LK are
funded by the Deutsche Forschungsgemeinschaft
under Germany's
Excellence Strategy EXC 2044-390685587,
Mathematics M\"unster: Dynamics-Geometry-Structure.
LK is also funded by the DFG through the Polish-German Beethoven Grant KR 1668/11.}}

\date{\it In Memoriam Jacques Tits}

\begin{document}

\maketitle

\begin{abstract}
We study solutions of a quadratic matrix equation arising in Riemannian geometry.
Let $S$ be a real symmetric $n\times n$-matrix with zeros on the diagonal and
let $\theta$ be a real number. 
We construct nonzero solutions $(S,\theta)$ of the set of quadratic equations \[\sum_kS_{i,k}=0\quad\text{ and }\quad\sum_{k}S_{i,k}S_{k,j}+S_{i,j}^2=\theta S_{i,j}\text { for }i<j.\]
Our solutions relate the equations to strongly regular graphs, to group rings, and to multiplicative characters of finite fields.
\end{abstract}

\section{Introduction}

In the work of B\"ohm and Wilking \cite{BW} 
on spherical space forms, or for the differentiable Sphere Theorem \cite{BS},
the Ricci flow and algebraic curvature tensors
play an important role. The Ricci flow is a solution to the PDE
\[\textstyle\frac{\partial}{\partial t}\langle-,-\rangle+2\langle\mathrm{Ric}(R)-,-\rangle=0,\]
for a time-dependent Riemannian metric $\langle-,-\rangle$ with Ricci tensor $\mathrm{Ric}(R)$ on a manifold $M$.
The curvature tensor $R$ is then also time-dependent, and satisfies the PDE
\[\textstyle\frac{\partial}{\partial t}R=\Delta R+2(R^2+R^\#).\] 
The right-hand side of this equation is a quadratic function of the curvature tensor $R$,
which is viewed as a selfadjoint endomorphism on $\bigwedge^2T_pM$.

In geometric applications of the Ricci flow such as \cite{BW}, one wants to show that the Riemannian metric
evolves towards a metric with constant positive sectional curvature, i.e. 
towards a metric where $R$ is a positive multiple of the identity,
assuming that the curvature tensor $R$ at time $t=0$ satisfies certain positivity assumptions.
Hamilton's maximum principle
asserts that certain positivity assumptions are preserved by the Ricci flow
if they are preserved by the ODE \[\textstyle\frac{d}{dt}R=R^2+R^\#.\]
It is then of particular interest to understand solutions $(R,\theta)$ of the 
quadratic equation
\[
 \theta R=R^2+R^\#.
\]
Such solutions $R$ can be viewed as obstructions to an evolution towards a metric with constant positive sectional curvature. 

This quadratic equation is difficult to handle, and one may impose additional
conditions on $R$. One possible assumption is that the metric is Einstein, i.e. that
the Ricci tensor $\mathrm{Ric}(R)$ is a scalar multiple of the identity. 
In addition, one can assume that 
$R$ has eigenvectors of a particularly simple form. In the present work,
we study $R$ under these two additional assumptions. This is a purely algebraic problem
about certain matrices. 

In his PhD thesis \cite{Jae}, J\"ager showed among other things that this equation has 
nontrivial solutions in dimensions $n\geq 4$ if the integer $n=m\ell$ is a composite number,
or if $n=p$ is a prime with $p\equiv1\pmod4$.
The latter solutions were constructed using Legendre symbols and certain finite fields.
The fact that a Riemannian problem is related to number theory caught the attention of
the first-named author and led to the present paper. 

We construct new solutions, using strongly regular graphs and 
multiplicative characters of finite fields.
As we wrote above, these solutions can be seen as obstructions to an
evolution towards a nice Riemannian metric. For most of them
it is not clear if they actually arise from Riemannian manifolds.
Most of the solutions that we construct have symmetries, i.e.
they come with natural actions of finite groups that permute the coordinates in $\RR^n$.
Our main result is as follows.

\begin{Thm}\label{MainA}
Let $V$ be a euclidean vector space with an orthonormal basis $e_1,\ldots,e_n$, with $n\geq 4$.
Then there exists an algebraic curvature tensor $R$ which is Einstein, such that
the $e_i\wedge e_j$ with $i<j$ are eigenvectors of $R$, and a real number $\theta\geq 0$ 
with 
\[
 R^2+R^\#=\theta R,
\]
and $R$ is not a multiple of the identity.

Such solutions $R$, where the eigenvalue spectrum of $R$ consists of precisely two real numbers,
are (up to scaling) in one-to-one correspondence with strongly regular graphs $\Gamma$.
If both $\Gamma$ and its complementary graph $\Gamma^c$ are connected, then $R$ is not the curvature
tensor of a product of spheres. 
In this case the dimension $n$ is the number
of vertices of the graph.

Other nontrivial solutions in dimension $n$ can be constructed from multiplicative characters of finite fields
$\mathbb F_q$, where $n=q$ is a prime power with $q\equiv 1\pmod m$ and  $m=3,4,8$.
\end{Thm}
The curvature tensor of a product of round spheres of dimensions at least $2$ 
(with suitable radii) gives an essentially trivial example
of a diagonal curvature tensor which is not a multiple of the identity, cp.~Example~\ref{SphereProductEx}
and \ref{CompleteGraph}.
Most of the solutions that we produce are of a different and nontrivial nature.
Theorem \ref{MainA} follows from Theorem~\ref{Main} below.

\section{The Riemannian background}

We describe how the equation arises in Riemannian geometry. We refer to \cite{Besse}
for the following general facts.

\begin{Num}
Let $(M,\langle-,-\rangle)$ be an $n$-dimensional Riemannian manifold, and let $p\in M$.
The \emph{Riemannian curvature tensor} $\sfR$ at $p$ is the trilinear map
\[(x,y,z)\longmapsto \sfR(x,y)z=[\nabla_x,\nabla_y]z-\nabla_{[x,y]}z\]
on the tangent space $T_pM$,
where $x,y,z$ are tangent vectors in $T_pM$ which are smoothly extended to vector fields on a small
neighborhood of $p$.
The curvature tensor satisfies the identities 
\begin{align*}
 \sfR(x,x)&=0\quad\text{($\sfR$ is alternating)}\\
 \sfR(x,y)+\sfR(x,y)^T&=0\quad\text{($\sfR(x,y)$ is skew-symmetric)}\\
\sfR(x,y)z+\sfR(z,x)y+\sfR(y,z)x&=0\quad\text{($\sfR$ satisfies the first Bianchi identity)}.
\end{align*}
The equations imply the interchange symmetry $\langle\sfR(x,y)u,v\rangle=\langle\sfR(u,v)x,y\rangle$.
One can turn $\sfR$ into a selfadjoint endomorphism $R$ on the second exterior power of the tangent space 
as follows.
Let $V$ denote an $n$-dimensional euclidean vector space with inner product $\langle-,-\rangle$
and orthogonal group $O(V)$.
The space of all selfadjoint endomorphisms of $V$ is denoted $S(V)$.
The inner product allows us to identify $V$ with its dual, and this identification is implicit in
some of the following notions.
We endow the second exterior power $\bigwedge^2V$ of $V$ with the
inner product \[\langle x\wedge y,u\wedge v\rangle=
\langle x,u\rangle\langle y,v\rangle-\langle x,v\rangle\langle y,u\rangle,\]
and the space of endomorphisms of $V$ with the inner product 
\[\langle A,B\rangle=\mathrm{tr}(A^TB).\]
An \emph{algebraic curvature tensor} $R$ is a selfadjoint endomorphism of $\bigwedge^2 V$ that satisfies the first Bianchi identity
\[
 \langle R(x\wedge y),z\wedge w\rangle+
 \langle R(z\wedge x),y\wedge w\rangle+
 \langle R(y\wedge z),x\wedge w\rangle=0.
\]
These endomorphisms form a linear subspace $S_B(\bigwedge^2V)\subseteq S(\bigwedge^2V)$.
In the case of the curvature tensor $\sfR$ of a Riemannian manifold,
we obtain an algebraic curvature tensor $R$ by putting
\[\langle R(x\wedge y),u\wedge v\rangle=-\langle \sfR(x,y)u,v\rangle.\]
Conversely, the left-hand side determines a trilinear map $\sfR$ having the symmetries
of the Riemannian curvature tensor. 
For linearly independent vectors $x,y\in V$,
the \emph{sectional curvature} of the $2$-dimensional subspace $H$ spanned by $x,y$
is \[K(H)=\frac{\langle R(x\wedge y),x\wedge y)}{\langle x\wedge y,x\wedge y\rangle}.\]
The choice of the sign in the definition guarantees that $R$ is positive
for a space form of constant positive sectional curvature.
\end{Num}
\begin{Num}
For the following results we refer to \cite{BW}, to \cite[Ch.~11]{Chowetal}, to
\cite[Ch.~1.G]{Besse} and to \cite[120.3]{GW}.
We recall that $S(V)$ denotes the space of selfadjoint endomorphisms of the euclidean vector space $V$,
and that the orthogonal group $O(V)$ acts on $S(V)$ by conjugation.
As an $O(V)$-module, $S(V)$ decomposes orthogonally into the simple module $S_0(V)$ consisting of the traceless 
selfadjoint endomorphisms, and a $1$-dimensional trivial module spanned by the identity map $\id_V$.
Similarly, one may decompose $S(\bigwedge^2 V)$ as an $O(V)$-module.
The resulting decomposition of the algebraic curvature tensor $R$ yields various curvature-related quantities.

First of all, we put $U_0=\RR\cdot\id_{\bigwedge^2V}$.
The traceless part $S_0(\bigwedge^2 V)$ can be decomposed orthogonally as an $O(V)$-module as
\[
 \textstyle S_0(\bigwedge^2 V)=U_1\oplus U_2\oplus U_3,
\]
with $U_3\cong\bigwedge^4V$.
A selfadjoint endomorphism $R\in S(\bigwedge^2V)$ satisfies the first Bianchi identity if and only if
$R\in U_0\oplus U_1\oplus U_2=S_B(\bigwedge^2V)$.

The module $U_1$ is isomorphic to $S_0(V)$ as follows.
The \emph{Ricci tensor} of $R$ is defined as
$\mathrm{Ric}(R)\in S(V)$ via 
\[\langle \mathrm{Ric}(R)u,v\rangle=\mathrm{tr}(x\longmapsto \sfR(x,u)v)=\sum_i\langle R(e_i\wedge u),e_i \wedge v\rangle,\]
where $e_1,\ldots,e_n$ is an orthonormal basis for $V$.
The trace of $\mathrm{Ric}(R)$ is the \emph{scalar curvature} $s(R)$, and 
$\mathrm{Ric}_0(R)=\mathrm{Ric}(R)-\frac{s(R)}{\dim(V)}\id_V$ is called the \emph{traceless Ricci tensor}.
The algebraic curvature tensor $R$ is called \emph{Einstein} if \[\mathrm{Ric}_0(R)=0.\]
The map $R\longmapsto\mathrm{Ric}_0(R)$ is $O(V)$-equivariant and its kernel is $U_0\oplus U_2\oplus U_3$. It maps
the $O(V)$-module $U_1$ isomorphically onto $S_0(V)$. 
An algebraic curvature tensor $R\in U_0\oplus U_1\oplus U_2=S_B(\bigwedge^2)$ decomposes thus as \[R=R_0+R_1+R_2.\]
The first summand $R_0=\frac{s(R)}{n(n-1)}\id_{\bigwedge^2V}$ encodes the scalar curvature, the second summand $R_1$ 
encodes the traceless part of the Ricci tensor, and the third summand $R_2=W(R)$ is called the \emph{Weyl curvature tensor}.
We note also that $s(R)=2\mathrm{tr}(R)$.
\end{Num}

\begin{Ex}[Round spheres]\label{SphereEx}
For $n\geq 2$, the curvature tensor of the round $n$-sphere $\Sphere^n(\rho)$ of radius $\rho>0$ is Einstein,
the sectional curvature $K(H)=\frac{1}{\rho^2}$ is constant, and $s(R)=\frac{n(n-1)}{\rho^2}$.
In particular, \[R=\frac{1}{\rho^2}\id_{\bigwedge^2V}\]
whence
$\mathrm{Ric}_0(R)=0$, $W(R)=0$, and $\mathrm{Ric}=\frac{n-1}{\rho^2}\id_V$.
\end{Ex}
If $R_V$ and $R_W$ are algebraic curvature tensors on euclidean vector spaces $V,W$, respectively,
then $R=R_V\oplus R_W$ acts in a natural way on $\bigwedge^2V\oplus\bigwedge^2W$,
and this action extends to  
\[\textstyle\bigwedge^2(V\oplus W)=\bigwedge^2V\oplus (V\otimes W)\oplus\bigwedge^2W\]
by mapping $V\otimes W$ to $0$.
For the scalar curvature we have $s(R)=s(R_V)+s(R_W)$, and 
the sectional curvature is $K(H)=0$ if $H\subseteq V\oplus W$ is a $2$-dimensional subspace spanned by a vector $v\in V$
and a vector $w\in W$. For the Ricci tensor we have
$\mathrm{Ric}(R)=\mathrm{Ric}(R_V)\oplus\mathrm{Ric}(R_W)$.
\begin{Ex}[Products of round spheres]\label{SphereProductEx}
Let $\rho,\sigma>0$ be real numbers, and $k,\ell\geq 2$ integers, with $\frac{k-1}{\rho^2}=\frac{\ell-1}{\sigma^2}=\tau$. 
Then the product $M=\Sphere^k(\rho)\times\Sphere^\ell(\sigma)$ is Einstein, with
scalar curvature $s(R)=(k+\ell)\tau$.
The tangent space at a point $(p,q)\in M$ decomposes as a sum $V\oplus W$, with 
$V=T_p\Sphere^k(\rho)$ and $W=T_q\Sphere^\ell(\sigma)$.
The curvature tensor $R$ has eigenvalue
$\frac{1}{\rho^2}$ on $\bigwedge^2V$, eigenvalue $0$ on $V\otimes W$, and eigenvalue
$\frac{1}{\sigma^2}$ on $W$. We put $n=k+\ell$.
Thus $R=R_0+W(R)$, with $R_0=\frac{\tau}{n-1}\id_{\bigwedge^2(V\oplus W)}$
and the Weyl curvature tensor
\[
W(R)=\frac{1}{n-1}\left(\frac{\ell}{\rho^2}\id_{\bigwedge^2V}\oplus (-\tau)\id_{V\otimes W}\oplus \frac{k}{\sigma^2}\id_{\bigwedge^2W}\right)
\]
is nontrivial.
\end{Ex}

\begin{Num}
Squaring is an $O(V)$-equivariant quadratic map on the space of selfadjoint operators.
In connection with his work on the Ricci flow, Hamilton introduced another quadratic map $\#$
as follows.
There is a natural $O(V)$-equivariant isomorphism
 $\iota:\bigwedge^2V\longrightarrow\mathfrak{so}(V)$
that maps $x\wedge y$ to the skew-symmetric endomorphism 
\[
\iota(x\wedge y):z\longmapsto x\langle y,z\rangle-y\langle x,z\rangle,
\]
and $\tr(\iota(x\wedge y)^T\iota(u\wedge v))=2\langle x\wedge y,u\wedge v\rangle$.
For a linear endomorphism $T$ of $\bigwedge^2V$ we denote its push-forward to
$\mathfrak{so}(V)$ by $\tilde T=\iota\circ T\circ\iota^{-1}$.
We recall that every element $X$ in the Lie algebra $\mathfrak{so}(V)$ 
determines a linear endomorphism
$\mathrm{ad}(X): Y\longmapsto [X,Y]$. This endomorphism $\mathrm{ad}(X)$ is skew symmetric
with respect to the Killing form of the Lie algebra $B(X,Y)=\tr(\mathrm{ad}(X)\circ\mathrm{ad}(Y))$.
In $\mathfrak{so}(V)$ we have $B(X,Y)=(n-2)\tr(XY)$, where $n=\dim(V)$.

If $R,T\in S(\bigwedge^2V)$ and $X,Y\in\mathfrak{so}(V)$, then the endomorphism
$\tilde R\circ \mathrm{ad}(X)\circ \tilde T\circ\mathrm{ad}(Y)$ is therefore 
selfadjoint.
We define an endomorphism $R\#T$ of $\bigwedge^2V$ via 
\[
\langle (R\#T)( x\wedge y),u\wedge v\rangle=-\frac{1}{2}\tr(\tilde R\circ \mathrm{ad}(X)\circ\tilde T\circ\mathrm{ad}(Y)),
\]
where $X=\iota(x\wedge y)$ and $Y=\iota(u\wedge v)$. From the properties of the trace we see that 
\[
 \langle (R\#T)( x\wedge y),u\wedge v\rangle=
 \langle (T\#R)(u\wedge v),x\wedge y\rangle=
 \langle (R\#T)( u\wedge v),x\wedge y\rangle.
\]
Hence $R\# T$ is selfadjoint, and $R\# T=T\# R$. We put $R^\#=R\# R$ and we note that $(\id_{\bigwedge^2V})^\#=(n-2)\id_{\bigwedge^2V}$.
If we choose an orthonormal basis of $\mathfrak{so}(V)$ consisting of matrices $Z_\xi$, 
we have 
\begin{align*}
 \langle (R\#T)( x\wedge y),u\wedge v\rangle 
 &= -\frac{1}{2}\sum_\xi \langle \tilde R[X,\tilde T[Y,Z_\xi]],Z_\xi\rangle \\
 &=\frac{1}{2}\sum_\xi \langle [Y,Z_\xi],\tilde T[X,\tilde RZ_\xi]\rangle \\
 &=\frac{1}{2}\sum_{\xi,\eta} \langle [Y,Z_\xi],Z_\eta\rangle\langle \tilde T Z_\eta,[X,\tilde RZ_\xi]\rangle \\
 &=\frac{1}{2}\sum_{\xi,\eta} \langle Y,[Z_\xi,Z_\eta]\rangle\langle X,[\tilde RZ_\xi,\tilde TZ_\eta]\rangle,
\end{align*}
which is the classical definition of $\#$, cp. \cite[1.3.1]{Jae}.

If $R\in S_B(\bigwedge^2 V)$ is decomposed into its components, $R=R_0+R_1+R_2$, then one can show that 
$
 R+\id_{\bigwedge^2V}\#R=(n-1)R_0+\frac{n-2}{n}R_1.
$
Moreover,
$
 (R^2+R^\#)_3=0.
$
Suppose that the algebraic curvature tensor $R$ is Einstein, $R=R_0+R_2$.
From the identities above we have
\begin{align*}
 R^2+R^\# & =(R_0)^2+(R_2)^2+(R_0)^\#+(R_2)^\#+2R_0(R_2+\id_{\bigwedge^2V}\#R_2)\\
 & = (n-1)(R_0)^2+(R_2)^2+(R_2)^\#,
\end{align*}
and one can show that $(R_2)^2+(R_2)^\#\in U_2$.
\end{Num}
\begin{Num}
We now consider a special type of algebraic curvature tensors.
Suppose that there is an orthonormal basis $e_1,\ldots,e_n$ of $V$ such that the 
vectors $e_i\wedge e_j$, for $i<j$, are eigenvectors of $R\in S(\bigwedge^2V)$,
\[
 R(e_i\wedge e_j)=r_{i,j}e_i\wedge e_j.
\]
Then $R$ satisfies the Bianchi identity and hence $R$ is an algebraic curvature tensor.
The $e_i$ are eigenvectors of the Ricci tensor, with $\mathrm{Ric}(R)e_i=\sum_{k}r_{i,k}e_i$,
and the scalar curvature is $s(R)=\sum_{i\neq j}r_{i,j}$. If we put $r=\frac{s(R)}{n(n-1)}$, then
$R_0=r\id_{\bigwedge^2V}$, and $R_1+R_2$ is diagonal, with eigenvalues $r'_{i,j}=r_{i,j}-r$.
Since $s(R)$ is the trace of $\mathrm{Ric}(R)$, the eigenvalues of $\mathrm{Ric}_0(R)$ are
$\sum_kr'_{i,k}$. In particular, the diagonal algebraic curvature tensor $R$ is Einstein if and only
if $\sum_kr'_{i,k}=0$ holds for all $i$. 

We may compute $R\# T$ explicitly if $R$ and $T$ are diagonal. For this, 
we adopt the following convention. For $i<j$ we put $\alpha=\{i,j\}$, 
$X_\alpha=\iota(e_i\wedge e_j)$, and $r_\alpha=r_{i,j}$, $t_\alpha=t_{i,j}$. 
The elements $\frac{1}{\sqrt2}X_\alpha$,
where $\alpha$ is a $2$-element subset of $\{1,\ldots,n\}$, form an orthonormal basis
of $\mathfrak{so}(V)$ with respect to the inner product $\langle X,Y\rangle=-\tr(XY)$.
We observe that $[X_\alpha,X_\beta]=0$ if $\alpha=\beta$ or $\alpha\cap\beta=\emptyset$.
If $\alpha\cap\beta$ is a singleton, then $[X_\alpha,X_\beta]=\pm X_\gamma$, where $\gamma=\alpha\triangle\beta$ is
the symmetric difference of $\alpha$ and $\beta$.
For $\alpha\neq\beta$ we have therefore
\[
 \sum_{\xi,\eta} \langle X_\alpha ,[X_\xi,X_\eta]\rangle\langle X_\beta,[\tilde RX_\xi,\tilde TX\eta]\rangle=
 \sum_{\xi,\eta} r_\xi t_\eta \langle X_\alpha ,[X_\xi,X_\eta]\rangle\langle X_\beta,[X_\xi,X\eta]\rangle=0,
\]
while
\[
 \sum_{\xi,\eta} \langle X_\alpha ,[X_\xi,X_\eta]\rangle\langle X_\alpha,[\tilde RX_\xi,\tilde TX\eta]\rangle=
 \sum_{\xi\triangle\eta=\alpha}r_\xi t_\eta\langle X_\alpha,[X_\xi,X_\eta]\rangle^2=
 4\sum_{\xi\triangle\eta=\alpha}r_\xi t_\eta.
\]
If one completes the $s_{i,j}$ and $r_{i,j}$ to symmetric matrices with $0$ on the diagonal, this shows that
\[
 (R\# T)(e_i\wedge e_j)=\frac{1}{2}\sum_k(r_{i,k}t_{j,k}+t_{i,k}r_{j,k}) e_i\wedge e_j.
\]
On the right-hand side we have thus the Jordan product of the symmetric matrices
$(r_{i,j})$ and $(t_{i,j})$. 

In any case, we have shown that $R\#T$ is again diagonal (and hence an algebraic
curvature tensor).
Suppose that $R=R_0+R_2$, with $R_0=r\id_{\bigwedge^2V}$ and that $R_2(e_i\wedge e_j)=w_{i,j}e_i\wedge e_j$,
with $\sum_kw_{i,k}=0$. Then $R_0R_2+R_0\# R_2=0$ and thus
\begin{align*}
 R^2+R^\# & = (R_0)^2+(R_2)^2+(R_0)^\#+(R_2)^\#\\
 & = (n-1)(R_0)^2+(R_2)^2+(R_2)^\#.
\end{align*}
Moreover, $\sum_j(w_{i,k}+\sum_kw_{i,k}w_{k,j})=0$. We put $S_{j,i}=S_{i,j}=w_{i,j}$ for $i<j$ and $S_{i,i}=0$.
The equation
\[
 R^2+R^\#=\theta R
\]
is thus equivalent to the equations
\begin{align*}
 (n-1)r^2 &=\theta r\\
 \sum_kS_{i,k} &=0 \tag{a}\\
 S_{i,j}^2+\sum_kS_{i,k}S_{k,j}&=\theta S_{i,j} \text{ for } i<j \tag{b}.
\end{align*}
We note that the $\ell_2$-norms of the matrix $S$ and the Weyl curvature tensor $R_2=W(R)$ are related by
\[
 ||S||_2=\sqrt{2}||W(R)||_2.
\]
In the Riemannian setting, one wants $s(R)=n(n-1)r>0$, and thus $\theta=\frac{s(R)}{n}>0$.
The case where all $S_{i,j}=0$ corresponds to the round sphere, as in Example~\ref{SphereEx}.
\end{Num}

\section{The matrix equations}

Let $S$ be a symmetric real $n\times n$ matrix with zeros on the diagonal.
It is easy to check that equation (a) has no solutions $S\neq 0$ for $n=2,3$.
We therefore assume that \[n\geq 4.\]
We rewrite equations (a) and (b) as matrix equations. 
Let $\J$ denote the all-one-matrix, which has entries $1$ everywhere
and let $\bfone$ denote the identity matrix.
We recall that the \emph{Hadamard product} of two matrices $X=(x_{i,j})$ and 
$Y=(y_{i,j})$ of the same shape is defined as the entry-wise product, \[X\odot Y=(x_{i,j}y_{i,j}).\]
Let $S$ be a symmetric real $n\times n$-matrix with zeros on the diagonal, and
let $D$ denote the diagonal matrix with entries $D_{i,i}=\sum_k S_{i,k}^2$.
The two matrix equations \[S\J=0\text{ and }S\odot S+S^2=\theta S+D\]
are equivalent to the equations (a), (b) above.
Thus we are interested in real $n\times n$-matrices $S$ that satisfy the four equations
\begin{equation}
\label{Basic}
 S=S^T,\quad S\odot \bfone=0,\quad S\J=0,\text{ and } S\odot S+S^2=\theta S+D,
\end{equation}
where $S^T$ is the transpose of $S$ and $D$ is some diagonal matrix.
If $S$ solves (\ref{Basic}), then $S$ has zeros on the diagonal and thus the 
diagonal entries of $D$ are given by $D_{i,i}=\sum_kS_{i,k}S_{k,i}$. Hence
\[\tr(D)=||S||_2^2.\]
Comparing the number of variables and equations in (\ref{Basic}), we expect a finite set of solutions $S\neq 0$ for a fixed value of $\theta$.
We note the following scale invariance.
If $S$ solves (\ref{Basic}) and if $t$ is a real number, 
then the matrix $\tilde S=tS$ solves the equations 
\begin{equation}
 \label{Scale}
\tilde S=\tilde S^T,\quad \tilde S\odot \bfone=0,\quad \tilde S\J=0,\text{ and } \tilde S\odot \tilde S+\tilde S^2=t\theta\tilde S+t^2D.
\end{equation}
If $S\neq 0$ solves (\ref{Basic}), a
scaling invariant measure for the size of $\theta$ is thus the quantity \[\hat\theta=\frac{|\theta|}{||S||_2}=\frac{|\theta|}{\sqrt{\tr(D)}}.\]
We observe the following.
First of all, every nonzero solution $S$ to the equations (\ref{Basic}) in dimension $n$
can be inflated to a nonzero solution $\begin{pmatrix} S & 0 \\ 0 & 0\end{pmatrix}$ in all higher dimensions $m>n$, for the same value of $\theta$. Also, we have the following. 
\begin{Lem}\label{Block}
Suppose that we have two nonzero solutions $S_1,S_2$ to the equations (\ref{Basic}), in dimensions $n_1$ and $n_2$ and for constants $\theta_1,\theta_2$, respectively.
If $\theta_1,\theta_2\neq0$,
then the block diagonal matrix $S=\frac{1}{\theta_1}S_1\oplus \frac{1}{\theta_2}S_2$ solves (\ref{Basic}) in dimension
$n$, with $\theta=1$ and \[n=n_1+n_2\text{ and }\hat\theta=\frac{\hat\theta_1\hat\theta_2}{\sqrt{\hat\theta_1^2+\hat\theta_2^2}}.\]
If $\theta_1=\theta_2=0$, then the block diagonal matrix $S=S_1\oplus S_2$ solves (\ref{Basic}) in dimension
$n$, with \[n=n_1+n_2\text{ and }\theta=\hat\theta=0.\]
\end{Lem}
One of our main results is as follows.
\begin{Thm}\label{Main}
The equations (\ref{Basic}) have a solution $S\neq 0$ in
every dimension $n\geq 4$, 
with $\hat\theta\neq0$.
A solution $S$ where the off-diagonal matrix entries of $S$ assume precisely two distinct
values exists if and only if $n=m\ell$ is a composite number, or if $n=p$ is a prime with $p\equiv1\pmod 4$.
In the latter case, $\hat\theta=0$.
\end{Thm}
\begin{proof}
If $n=m\ell$, with $\ell,m\geq 2$, then Examples~\ref{CompleteGraph} or \ref{Composite} provide solutions
with $\hat\theta\neq 0$. In particular, we have a solution in dimension $4$,
with $\hat\theta\neq0$. By the remark preceding Lemma~\ref{Block}, there is
a solution in every dimension $n\geq 4$.

A different, and more geometric argument runs 
as follows. Put $n=n_1+n_2$, with $n_1,n_2\geq 2$.
Then the product of two round spheres of dimensions $n_1,n_2$ as in Example \ref{SphereProductEx} provides
a solution with $\hat\theta\neq 0$. These specific solutions have typically three different
eigenvalues for the curvature tensor $R$, and thus the off-diagonal elements of $S$ 
assume three different values.

%
The last claim follows from Theorem~\ref{MainThm1}.
\end{proof}

\section{Adjacency matrices of graphs}

For the material in this section we refer to \cite{Hall}, \cite{GodRoy} and \cite{BVM}. 
Suppose that $\Gamma$ is a finite simplicial graph, with vertex set $V$.
Then $\Gamma$ is encoded in the \emph{adjacency matrix} $A\in\RR^{V\times V}$,
which has entries 
\[
 A_{u,v}=\begin{cases} 1&\text{if }u,v\text{ are adjacent}\\ 0 &\text{else.}\end{cases}
\]
The matrix $A$ is symmetric and has entries $\{0,1\}$, with zeros on the diagonal.
In other words, $A$ satisfies the three equations
\[
 A=A^T,\quad A\odot\bfone=0,\text{ and } A\odot A=A.
\]
Conversely, every matrix with real entries satisfying these equations is the adjacency matrix of
a finite simplicial graph, unique up to isomorphism.
For example, the matrix \[\K=\J-\bfone\] is the adjacency matrix of the complete graph.
On the other extreme, $A=0$ is the adjacency matrix of a graph with no edges.

The \emph{complementary graph} $\Gamma^c$ has an edge between two distinct vertices
$u,v$ if and only if there is no edge in $\Gamma$ between $u$ and $v$. The adjacency
matrix of $\Gamma^c$ is thus \[A^c=\K-A.\]
\begin{Def}
A graph $\Gamma$ on $n$ vertices is \emph{$r$-regular} if every vertex has $r$ neighbors.
This translates into the matrix equation 
\[
 A\J=r\J
\]
for its adjacency matrix $A$.
\end{Def}
For $k\geq 1$, the $(u,v)$-entry of the matrix power $A^k$ counts the number of simplicial paths of length 
$k$ from the vertex $u$ to the vertex $v$.
Hence the graph is $r$-regular if and only if $A^2$ has all diagonal entries equal to $r$,
\[
 A^2\odot\bfone=r\bfone.
\]
\begin{Def}\label{srgdef}
A simplicial graph $\Gamma$ on $n$ vertices is \emph{strongly regular} if there are integers
$r,\lambda,\mu\geq 0$ such that the following hold.
\begin{enumerate}[\rm(i),nosep]
\item Every vertex has $r$ neighbors.
 \item Any two adjacent vertices have $\lambda$ common neighbors.
 \item Any two non-adjacent vertices have $\mu$ common neighbors.
\end{enumerate}
Then $\Gamma$ is called an $\srg(n,r,\lambda,\mu)$ for short. For the adjacency matrix $A$, this translates into the
equation
\[
           A^2=r\bfone+\lambda A+\mu A^c=r\bfone+(\lambda-\mu)A+\mu\K.
          \]
Indeed, the matrix entry $(A^2)_{u,v}$ counts the number of common neighbors of the two vertices $u,v$.
This yields immediately the equation for $A^2$.
In particular, a real $n\times n$ matrix $A$ is the adjacency matrix of an $\srg(n,r,\lambda,\mu)$ if and only if it satisfies the equations
\begin{equation}
\label{srgeqs}
A=A^T,\quad A\odot\bfone=0,\quad A\odot A=A,\text{ and }
 A^2=(\lambda-\mu)A+\mu\J+(r-\mu)\bfone,
\end{equation}
for real numbers $r,\lambda,\mu$.
\end{Def}
The complementary graph of an $\srg(n,r,\lambda,\mu)$ is an
$\srg(n,r^c,\lambda^c,\mu^c)$, with 
\[
r^c+r=n-1,\quad \lambda^c-\mu=n-2(r+1), \text{ and }\mu^c-\lambda=n-2r,
 \]
as is easily checked.

Our definition of a strongly regular graph allows for some degeneracies, such as $r=0$.
If adjacency (together with equality) is a nontrivial equivalence relation, then $\lambda=r-1$ and the graph is a finite disjoint union of $k$ complete graphs on $r+1$ vertices.
Then the adjacency matrix $A$ is a block diagonal matrix, with $k$ blocks of size $r+1$ with $1$ as off-diagonal entries.
The complementary graph corresponds to the situation when
non-adjacency is a nontrivial equivalence relation. Then the graph is
a complete $k$-partite graph, where each vertex set of a given color has size $r+1$.
We refer to \cite[Ch.~20]{GodRoy}.
A strongly regular graph $\Gamma$ is called \emph{primitive} if both $\Gamma$ and
$\Gamma^c$ are connected.

Suppose that $\Gamma$ is an $\srg(n,r,\lambda,\mu)$, and that $u$ is a vertex in $\Gamma$.
If $w$ is a vertex at distance $2$ from $u$, then there are $\mu$ paths of length $2$ connecting
$u$ and $w$. The number of vertices at distance $2$ from $u$ is then $n-1-r$, and hence
there are $\mu(n-1-r)$ paths of length $2$ connecting $u$ to vertices at distance $2$.
We may compute this number in another way. If $v$ is at distance $1$ from $u$, then 
$v$ is adjacent to $\lambda$ other vertices at distance $1$ from $u$. Hence $v$ is adjacent
to $r-\lambda-1$ vertices at distance $2$ from $u$, whence
\begin{equation}
\label{Relation}
 \mu(n-r-1)=r(r-\lambda-1).
\end{equation}
We derived this formula under the assumption that there are vertices at distance $2$.
A direct inspection shows, however, that it remains valid for the two exceptional cases 
$r=0$ and $r=\lambda+1$,
hence it is valid for all strongly regular graphs in our sense.

\section{The case of two matrix entries}

In this section we consider the solutions to (\ref{Basic}) for the special case that 
the off-diagonal entries of $S$ assume at most two different values.
If they assume exactly one value $x$, then $S=x\K$ and thus $S\J=x\K\J=x(n-1)\J$, whence $x=0$.
It remains to consider the case where the off-diagonal entries assume two different values $x\neq y$.
The matrix $S$ is thus of the form
\[
 S=xA+yA^c=(x-y)A+y\K=zA+y\K,
\]
where $A\neq 0,\K$ is an adjacency matrix of a graph $\Gamma$ on $n$ vertices, and $z=x-y\neq 0$. 
In view of the scaling invariance (\ref{Scale}) we may put $z=1$.
We note that 
\[\J^2=n\J,\quad \K\J=\J\K=(n-1)\J\text{ and }\K^2=(n-2)\K+(n-1)\bfone.\]
The condition $S\J=0$ becomes 
\begin{equation}
 \label{2ValuesReg}
 0  = S\J = A\J+y\K\J= A\J+y(n-1)\J.
\end{equation}
This holds if and only if $\Gamma$ is $r$-regular, for \[r=y(1-n).\]
Next we compute
\begin{align*}
 S\odot S& =A\odot A+2yA\odot\K+y^2\K\odot \K \\
 &=  A+2yA+y^2\K.
\end{align*}
Assuming $S\J=0$, we have
\begin{align*}
S^2 & =A^2+y(A\K+\K A)+y^2\K^2 \\
& = A^2+y(A\J+\J A-2A)+y^2(n-2)\K+y^2(n-1)\bfone\\
&= A^2+2yr\J-2yA+y^2(n-2)\K+y^2(n-1)\bfone\\
&= A^2-2yA-y^2n\K+y^2(1-n)\bfone
\end{align*}
and therefore
\begin{equation}
\label{2ValuesEq}
 S\odot S+S^2  =A^2+A+yr\K+yr\bfone.
\end{equation}
\begin{Thm}\label{MainThm1}
Suppose that the $n\times n$ matrix $S$ solves (\ref{Basic}) and that the off-diagonal entries
of $S$ assume precisely two different values $x\neq y$. Put $z=x-y$.
Then \[S=zA+y\K\]
and $A\neq0,\K$ is the adjacency matrix of a strongly regular graph $\Gamma$.

Conversely, if $A\neq 0,\K$ is the adjacency matrix of an $\srg(n,r,\lambda,\mu)$, then $S=zA+y\K$ solves
(\ref{Basic}), with 
\[\frac{y}{z}=\frac{r}{1-n}\quad\text{ and }
\quad\hat\theta=|\lambda-\mu+1|\sqrt{\frac{n-1}{nrr^c}},
 \]
 where $r^c=n-r-1$.
 
An integer $n$ occurs in this situation if and only if $n$ is either a composite number or
if $n=p$ is a prime with $p\equiv1\pmod4$. In the latter case $\theta=0$.
\end{Thm}
We note that the passage from $\Gamma$ to the complementary graph $\Gamma^c$ amounts to a sign change of $\theta$,
since $\lambda^c-\mu^c+1=-(\lambda-\mu+1)$.
\begin{proof}
Suppose that $S=A+y\K$ solves (\ref{Basic}), where $A\neq 0,\K$ is the adjacency matrix of a graph
$\Gamma$ on $n$ vertices. Then (\ref{2ValuesReg}) shows that $\Gamma$ is  $r$-regular, with $r=y(1-n)$. In particular, the matrix $A^2$ has $r$ on its diagonal.
By (\ref{2ValuesEq}) we have
\[
 A^2=(\theta -1)A+y(\theta -r)\K+D-yr\bfone.
\]
Hence $D=r(y+1)\bfone$ and $A$ is the adjacency matrix of
a strongly regular graph. Because of the scaling invariance (\ref{Scale}) this conclusion
remains valid for all other values of $z\neq0$.

Assume now that $A\neq 0,\K$ is the adjacency matrix of an $\srg(n,r,\lambda,\mu)$, whence
\[A^2=(\lambda-\mu)A+\mu\K+r\bfone.\] If we put $y=r/(1-n$),
then $A+y\K$ solves (\ref{Basic}), with  \[\theta=\lambda-\mu+1.\]
Indeed, $\mu=y(\theta-r)$ holds by (\ref{Relation}).
We have then $\tr(D)=nr\frac{n-r-1}{n-1}=\frac{nrr^c}{n-1}$, which yields the expression for $\hat\theta$.

If the integer $n=m\ell$ is decomposable, then Example~\ref{Composite} provides the
existence of a non-primitive strongly regular graph on $n$ vertices, while the Paley graphs
in Example~\ref{Palyegraph} provide examples for $n=p$, when $p\equiv 1\pmod 4$ is a prime.
A strongly regular graph whose number of vertices is a prime $p$ is a \emph{conference graph} by
\cite[10.3.4]{GodRoy} and then ${p\equiv1\pmod 4}$ by \cite[10.3.2]{GodRoy}. A conference graph
is an $\srg(n,\frac{n-1}{2},\frac{n-5}{4},\frac{n-1}{4})$, whence $\theta=0$.
\end{proof}
It is therefore of some interest to consider the parameters of  strongly regular graphs
and the resulting solutions to (\ref{Basic}).
The first example arises from non-primitive strongly regular graphs.
It corresponds to the Weyl curvature tensor of a product of $m$ round spheres of dimension $\ell$.
\begin{Ex}[Disjoint unions of complete graphs]\label{CompleteGraph}
Let $\Gamma$ be a graph which is a disjoint union of $m$ complete graphs on $\ell$ vertices.
Then $\Gamma$ is an $\srg(m\ell,\ell-1,\ell-2,0)$. For $m\geq 2$ we obtain solutions with 
\[
 n=m\ell,\qquad
 \theta=\ell-1,\qquad
 \hat\theta=\frac{1}{\ell}\sqrt{\frac{(\ell-1)(m\ell-1)}{m(m-1)}}.
\]
\end{Ex}
Primitive strongly regular graphs ($r,\mu\neq 0$) exist only for very specific values for the
number of vertices $n$, see eg.~\cite{BVM}.
We mention a few examples.
\begin{Ex}[Kneser Graphs $K(m,2)$]\label{Knesergraph}
The vertex set of the Kneser graph $K(m,2)$ is the set of transpositions $(i,j)$ in the symmetric group $\mathrm{Sym}(m)$.
Two transpositions are adjacent in $\Gamma=K(m,2)$ if they commute. Thus $\Gamma$ has 
$n=\binom{m}{2}$ vertices, and $\Gamma$ is $r$-regular for $r=\binom{m-2}{2}$. 
Two commuting transpositions commute with 
$\lambda=\binom{m-4}{2}$ other transpositions and two non-commuting transpositions commute
with $\mu=\binom{m-3}{2}$ other transpositions.
Thus $\Gamma$ is an $\srg(\binom{m}{2},\binom{m-2}{2},\binom{m-4}{2},\binom{m-3}{2})$, and
for $m\geq4$ we obtain solutions with
\[
n=\binom{m}{2},\qquad
\theta=5-m,\qquad
\hat\theta=|5-m|\sqrt{\frac{m+1}{m(m-1)(m-2)(m-3)}}.
\] 
\end{Ex}
There exist also sporadic examples of strongly regular graphs.
\begin{Ex}[Fischer's group $\mathrm{Fi}_{24}$]
The Fischer group $\mathrm{Fi}_{24}$, which is one of the sporadic finite simple groups discovered by B.~Fischer,
acts as a permutation group of degree $3$ on a set $V$ of cardinality $306936$.
The action of $\mathrm{Fi}_{24}$ on $V\times V$ has three orbits $D,E,F$, one of which is the diagonal $D$.
The orbit $E$ can be viewed as an edge set of a graph with vertex set $V$.
In this way, one obtains a strongly regular graph $\Gamma$, which is an
$\srg(306936, 31671, 3510, 3240)$. 
Hence we have
\[
 n=306936,\qquad
\theta=271,\qquad
\hat\theta=\frac{271}{3024}\sqrt{\frac{785}{748374}}\approx 2.9024382\cdot 10^{-3}.
\]
We refer to \cite{BVM} for more details.
\end{Ex}
\begin{Ex}[Collinearity graphs of finite generalized quadrangles]\label{GQgraph}
Let $\cQ$ be a finite generalized quadrangle of order $(s,t)$. Hence there are no digons in
$\cQ$, every line is incident with 
$s+1$ points and every point is incident with $t+1$ lines. For every non-collinear point-line pair
$(a,\ell)$, there is a unique line $h$ through $a$ and a unique point $b$ on $\ell$ such that
$b$ is on $h$.
Examples arise from symplectic forms in dimension $4$ and from nondegenerate quadratic and hermitian forms of Witt index $2$
over finite fields.
These examples have $s=t$ or $s=t^2$ or $s^2=t$, where $s$ is a prime power.
The \emph{collinearity graph} $\Gamma$ of $\cQ$ has the points of $\cQ$ as vertices, and two points are adjacent if they
are collinear. The number of points is $n=sts+ts+s+1=(s+1)(ts+1)$ and every point is collinear with 
$r=s(t+1)$ points. Two collinear points are collinear with $s-1$ other points,
and two non-collinear points are collinear with $t+1$ points.
Hence $\Gamma$ is an $\srg((s+1)(ts+1),s(t+1),s-1,t+1)$, and
\[
n=(s+1)(st+1),\qquad
\theta=s-t-1,\qquad
 \hat\theta =\frac{|s-t-1|}{s}\sqrt{\frac{st+t+1}{(s+1)(ts+1)t(t+1)}} .
\]
\end{Ex}
The next two examples are Cayley graphs.
\begin{Ex}[The rook's graph] \label{Rooksgraph}
The vertices of the \emph{rook's graph} $\Gamma$ are the fields of an $m\times m$ chessboard, for $m\geq 2$.
Two vertices are adjacent if the rook can pass from one to the other in one move.
Hence $\Gamma$ is $2(m-1)$-regular. Every pair of adjacent vertices 
has $m-2$ common adjacent vertices and every pair of non-adjacent vertices has $2$ common
adjacent vertices. Thus $\Gamma$ is an $\srg(m^2,2(m-1),m-2,2)$ and
\[
n=m^2,\qquad
\theta=m-3,\qquad
 \hat\theta=\frac{|m-3|}{m(m-1)}\sqrt{\frac{m+1}{2}} .
\]
We note that the rook's graph is a Cayley graph for the group $G=\ZZ/m\times\ZZ/m$ for the symmetric generating set 
$\{(i,j)\in G\mid (i,j)\neq(0,0)\text{ and } i=0\text{ or }j=0\}$.
\end{Ex}
The following strongly regular graphs were discovered by Paley~\cite{Paley}.
\begin{Ex}[Paley Graphs]\label{Palyegraph}
Let $\FF_q$ denote the finite field of order $q$, where $q\equiv1\pmod 4$ is a prime power.
Let $Q\subseteq\FF_q^\times$ denote the set of nonzero squares. Then $Q$ is the unique subgroup of index $2$ in the multiplicative
group $\FF_q^\times$, and $-1\in Q$. 
Let $\Gamma$ denote the Cayley graph of the additive group $\FF_q$ with respect to the symmetric generating set 
$Q$. Then $\Gamma$ is $r$-regular, with $r=\frac{q-1}{2}$.
We note that the additive group of $\FF_q$ acts transitively on $\Gamma$.
The number of common neighbors of $0$ and $c\neq0 $ is the number of all field elements $x^2\neq 0,c$ 
of the form $x^2-c=y^2$.
The number of solutions $(x,y)$ to the equation $x^2-y^2=c$ is $q-1$.
Hence there are $\frac{q-1}{4}$ common neighbors if $c$ is not a square.
If $c$ is a square, then there are $\frac{q-1}{4}-1$ common neighbors.
Therefore the Paley graph 
$\Gamma$ is an $\srg(q,\frac{q-1}{2},\frac{q-5}{4},\frac{q-1}{4})$, and
\[
n=q,\qquad
 \theta=\hat\theta=0.
\]
J\"ager rediscovered the Paley graphs in his thesis~\cite[Thm.~2.2.1]{Jae} for the case where $q\equiv1\pmod4$ is a prime, using Legendre symbols. 
The article \cite{Jones} provides many interesting details about Paley graphs and their history.
\end{Ex}
The last two examples were Cayley graphs. Strongly regular Cayley graphs can be described in a uniform way as follows.
\begin{Def}
Let $G$ be a finite group of order $n$ and let $S=S^{-1}$ be a symmetric subset not containing the identity.
The \emph{Cayley graph} $\Gamma=\Gamma_{G,S}$ has $G$ as its vertex set. Two vertices $a,b\in G$ are adjacent if and only if
$a^{-1}b\in S$. Hence $\Gamma$ is $r$-regular, with $r=|S|$. The group $G$ acts on $\Gamma$ via its natural left regular action on itself.
The common neighbors of $e$ and $a\neq e$ are the elements $s\in S$ with $a^{-1}s=t\in S$.
Hence $\Gamma$ is an $\srg(n,r,\lambda,\mu)$ if and only if the equation $a=st^{-1}$ has $\lambda$ solutions $(s,t)\in S\times S$ for all $a\in S$,
and $\mu$ solutions for all $a\not\in S\cup\{e\}$. In this case $S$ is called a \emph{regular partial difference set}.
The survey article \cite{Ma} contains many examples of regular partial difference sets.
In \cite{Ott1,Ott2}, regular partial difference sets in additive groups of finite fields are studied, using properties of Jacobi sums.

The following construction generalizes the rook's graph.
\end{Def}
\begin{Ex}[Families of subgroups]\label{PDS}
Let $(A,+)$ be a finite abelian group of order $n=m^2$ and let $H_1,\ldots,H_\ell$ be subgroups of order $m$,
with $H_i\cap H_j=\{0\}$ for $i\neq j$. We put $S=\bigcup_{i=1}^\ell H_i -\{0\}$ and we note that $\ell\leq m+1$.
For $i\neq j$ we have $A=H_i\oplus H_j$.
If $a\in A$, then the equation $a=s-t$ has $\ell(\ell-1)$ solutions $(s,t)\in S\times S$ for $a\not\in S\cup\{0\}$,
and $\ell^2-3\ell+m$ solutions for $a\in S$.
Hence $\Gamma_{A,S}$ is an $\srg(m^2,\ell(m-1),\ell^2-3\ell+m,\ell^2-\ell)$ and for $1\leq \ell\leq m$ we obtain solutions with
\[
 n=m^2,\quad
 \theta=m-2\ell+1,\quad
 \hat\theta=\frac{|m-2\ell+1|}{m}\sqrt{\frac{m+1}{\ell(m-1)(m-\ell+1)}}.
\]
For example, we may put $A=\FF_q\oplus\FF_q$ for any finite field $\FF_q$, and $H_i=\{(x,m_ix)\mid x\in\FF_q\}$, for $\ell$ distinct elements $m_1,\ldots,m_\ell\in\FF_q$.
For $\ell=2$ we recover the rook's graph.
\end{Ex}

\section{Solutions from group rings}

Let $G$ be a finite group. For $\KK=\RR,\CC$, let $R(G,\KK)=\KK^G$ denote the commutative ring of all $\KK$-valued functions on $G$.
Since $G$ is finite, the evaluation maps $ev_g:\phi\longmapsto\phi(g)$, for $g\in G$, span the dual space of $R(G,\KK)$.
Hence the dual space of $R(G,\KK)$ can be naturally identified with the underlying vector space of the \emph{group ring} $\KK[G]$, which consists
of finite linear combinations of group elements.
The multiplication in $\KK[G]$ induces via this duality a \emph{comultiplication} $\Delta:R(G,\KK)\longrightarrow R(G,\KK)\otimes R(G,\KK)\cong\KK^{G\times G}$, 
\[\Delta(\phi)(a,b)=\phi(ab).\] 
We define the \emph{symmetry} $\sigma$ and the \emph{augmentation} $\epsilon$ on $R(G,\KK)$ as
\[
 \phi^\sigma(g)=\phi(g^{-1})\quad\text{and}\quad\epsilon(\phi)=\phi(e).
\]
These data turn $R(G,\KK)$ into a symmetric \emph{Hopf algebra}. We refer to~\cite{HM} Theorem 3.76 and the material in this section.
Since $G$ is finite, the standard bilinear form
\[
 \langle\phi,\psi\rangle = \sum_{g\in G}\phi(g)\psi(g)
\]
on $R(G,\KK)$ yields a vector space isomorphism $\iota$ between $R(G,\KK)$ and its dual $\KK[G]$,
with 
\[
 \iota(\phi)=\sum_{g\in G}\phi(g)g.
\]
Under this isomorphism, the (possibly noncommutative) product in the group ring $\KK[G]$ translates into
the \emph{convolution product} $*$ on $R(G,\KK)$, which is given by
\[(\phi *\psi)(g)=\sum_{h}\phi(h)\psi(h^{-1}g).\]
\begin{Thm}\label{MainThm2}
Let $G$ be a finite group of order $n$ and let $\phi\in R(G,\RR)$. 
We define a matrix $S\in\RR^{G\times G}$ via \[S_{a,b}=\Delta(\phi)(a^{-1},b)=\phi(a^{-1}b).\]
Then $S$ satisfies (\ref{Basic}) for a real constant $\theta$ if and only if $\phi$ satisfies the equations 
\begin{equation}\label{Hopf}
 \phi=\phi^\sigma,\quad \epsilon(\phi)=0,\quad  \langle\phi,1\rangle=0,\quad \phi^2+\phi*\phi=\theta\phi+||\phi||_2^2\delta_e,
\end{equation}
where $\delta_e$ is Kronecker's $\delta$-function, $\delta_e(g)=\delta_{e,g}$.
For the solution $S$ we have then 
\[
 n=|G|,\quad\hat\theta=\frac{|\theta|}{\sqrt{n}||\phi||_2}.
\]
\end{Thm}
\begin{proof}
The first equation in (\ref{Hopf}) expresses the symmetry $S=S^T$, the second equation expresses the fact 
that $S_{a,a}=0$ for all $a$ and the third equation says that the sum over every row in the matrix $S$ is zero.
The matrix entries in $S\odot S$ are given by \[(S\odot S)_{a,b}=\phi^2(a^{-1}b)=\Delta(\phi^2)(a^{-1},b).\]
The matrix entries in $S^2$ are given by 
\[
 (S^2)_{a,b}=\sum_c\phi(a^{-1}c)\phi(c^{-1}b)=\sum_g\phi(g)\phi(g^{-1}a^{-1}b)=(\phi*\phi)(a^{-1}b)=\Delta(\phi*\phi)(a^{-1},b).
\]
In particular, $(S^2)_{a,a}=(\phi*\phi)(e)=\langle\phi,\phi^\sigma\rangle$ and
hence $||S||_2^2=n||\phi||_2^2$ if $\phi=\phi^\sigma$.
\end{proof}
In a completely analogous fashion one proves the following result.
\begin{Thm}\label{MainThm3}
Let $G$ be a finite group of order $n$ and let $\alpha\in R(G,\RR)$. For $a,b\in G$ put \[A_{a,b}=\Delta(\alpha)(a^{-1},b)=\alpha(a^{-1}b).\]
Then $A$ is the incidence matrix of an $\srg(n,r,\lambda,\mu)$ if and only if $\alpha$ satisfies the equations 
\begin{equation}\label{Hopfsrg}
 \alpha=\alpha^\sigma,\quad \epsilon(\alpha)=0,\quad  \alpha^2=\alpha,\quad \alpha*\alpha=(\lambda-\mu)\alpha+\mu+(r-\mu)\delta_e.
\end{equation}
\end{Thm}
This result corresponds to (\ref{srgeqs}).
If (\ref{Hopfsrg}) holds, then the function 
\[
 \phi=\alpha+\frac{r}{1-n}(1-\delta_e)
\]
solves (\ref{Hopf}) by Theorem~\ref{MainThm1}, with $\hat\theta=|\lambda-\mu+1|\sqrt{\frac{n-1}{nrr^c}}$ and $r^c=n-r-1$.
\begin{Rem}
We can translate this result into the group ring $R(G,\RR)$, using the isomorphism $\iota:R(G,\RR)\longrightarrow\RR[G]$
as follows. The map $\alpha\in R(G,\RR)$ satisfies the conditions (\ref{Hopfsrg}) if and only if the element 
$a=\iota(\alpha)=\sum_g\alpha(g)g$ satisfies the conditions
\[
 a=a^\sigma,\quad
 \epsilon(a)=0,\quad
 a\odot a=a,\quad
 a^2=(\lambda-\mu)a+\mu w+(r-\mu)e,
\]
where $\big(\sum_ga_gg\big)^\sigma=\sum_ga_gg^{-1}$, $\epsilon\big(\sum_ga_gg\big)=a_e$,
$\big(\sum_ga_gg\big)\odot\big(\sum_hb_hh\big)=\sum_ga_gb_gg$ and $w=\sum_gg$.
These conditions can be found in the literature, see eg.~\cite[Thm.~1.3]{Ma}.
\end{Rem}

Suppose that $(G,+)$ is a finite abelian group. For subsets $A,B\subseteq G$ with characteristic functions 
$\chi_A,\chi_B$ we have 
\[
 (\chi_A*\chi_B)(g)=|A\cap (g-B)|.
\]
The following generalizes Examples~\ref{CompleteGraph} (disjoint unions of complete graphs) and \ref{Rooksgraph} (the rook's graph).
\begin{Ex}\label{Composite}
Let $L,M$ be finite abelian groups of orders $\ell,m\geq 2$ and put $G=L\times M$.
We put
\begin{align*}
 A&=\{(a,0)\in G\mid a\neq 0\}\\
 B&=\{(0,b)\in G\mid b\neq 0\}\\
 C&=\{(a,b)\in G\mid a,b\neq 0\}\\
\end{align*}
and we consider the function
\begin{align*}
\phi=s\chi_A+t\chi_B+\chi_C
\end{align*}
for real parameters $s,t$.
We have $\phi(0,0)=0$, $\phi^\sigma=\phi$, and
\[
 \langle\phi,1\rangle=s(\ell-1)+t(m-1)+(\ell-1)(m-1).
\]
Moreover,
\[
 \phi^2=s^2\chi_A+t^2\chi_B+\chi_C.
\]
We obtain
\[
 \phi*\phi=s^2\chi_A*\chi_A+t^2\chi_B*\chi_B+\chi_C*\chi_C+2st\chi_A*\chi_B+2t\chi_B*\chi_C+2s\chi_C*\chi_A.
\]
 For $a,b\neq 0$ we find that 
 \begin{align*}
 (\phi^2+\phi*\phi)(0,0)&=s^2(\ell-1)+t^2(m-1)+(\ell-1)(m-1)\\
  (\phi^2+\phi*\phi)(a,0)&=s^2(\ell-1)+(\ell-2)(m-1)+2t(m-1) \\
  (\phi^2+\phi*\phi)(0,b)&=t^2(m-1)+(\ell-1)(m-2)+2s(\ell-1)\\
  (\phi^2+\phi*\phi)(a,b)&=1+(\ell-2)(m-2)+2st+2s(\ell-2)+2t(m-2),
 \end{align*}
 whence 
 \begin{align*}
 \phi^2+\phi*\phi &= (s^2(\ell-1)+t^2(m-1)+(\ell-1)(m-1))\delta_{(0,0)} \\
 &\quad +(s^2(\ell-1)+(\ell-2)(m-1)+2t(m-1))\chi_A\\
 &\quad +(t^2(m-1)+(\ell-1)(m-2)+2s(\ell-1))\chi_B\\
 &\quad +(1+(\ell-2)(m-2)+2st+2s(\ell-2)+2t(m-2))\chi_C.
 \end{align*}
 The equation $\phi^2+\phi*\phi=\theta\phi+u\delta_{(0,0)}$ then has the following solutions,
 for $n=m\ell$.
 \begin{center}
 \begin{tabular}{c|c|c|c}
 $s$ & $t$ & $\theta$  & $\hat\theta$ \\ \hline
  $\vphantom{\bigg|}\frac{1-m}{2}$ & $\frac{1-\ell}{2}$ & $\frac{4-(m-1)(\ell-1)}{2}$ & $\frac{|4-(m-1)(\ell-1)|}{\sqrt{m\ell(\ell-1)(m-1)(m+\ell+2)}}$ \\ \hline
  $\vphantom{\bigg|}\frac{\ell(1-m)}{\ell-1}$ & $1$ & $1-m\ell$ & $\frac{1}{\ell}\sqrt{\frac{(\ell-1)(m\ell-1)}{m(m-1)}} $\\ \hline
  $\vphantom{\bigg|}1$ & $\frac{m(1-\ell)}{m-1}$ & $1-m\ell$ & $\frac{1}{m}\sqrt{\frac{(m-1)(m\ell-1)}{\ell(\ell-1)}} $ 
 \end{tabular}
\end{center}
The first of these solutions is also contained in J\"ager's Thesis~\cite[Thm.~2.2.5]{Jae}.
For $m=\ell$ we recover (the negative of) the solution given by the rook's graph \ref{Rooksgraph}. 
The second and third solution correspond to Example~\ref{CompleteGraph}.
\end{Ex}

\section{Solutions from multiplicative characters}
\label{MultCharSection}
We recall some facts about finite (hence compact) abelian groups. Chapter 7 in \cite{HM} is an excellent reference for this material.
We note that we invoke here Pontrjagin duality and the Peter-Weyl Theorem for finite abelian groups, where they are 
readily proved by simple means.
\begin{Facts}
Let $A$ be a finite abelian group. A \emph{character} is a homomorphism \[\alpha:A\longrightarrow\Sphere^1\subseteq\CC^\times.\]
The characters of $A$ form under pointwise multiplication an abelian group $X(A)$, the \emph{character group} or \emph{Pontrjagin dual} of $A$.
We denote the trivial character by $\eps$. The multiplicative inverse of a character $\alpha$ is its complex conjugate $\bar\alpha$.
The duality pairing $X(A)\times A\longrightarrow\Sphere^1$
allows us in particular to view the elements of $A$ as characters for $X(A)$, and
Pontrjagin duality asserts that the natural homomorphism
$A\longrightarrow X(X(A))$ is an isomorphism. Since $A$ is finite, there is also a (non-natural) isomorphism
$A\cong X(A)$.
The \emph{annihilator} of a subset $B\subseteq A$ is the subgroup
\[
 B^\perp=\{\alpha\in X(A)\mid \alpha(b)=1\text{ for all }b\in B\}.
\]
By the annihilator mechanism, restriction of characters provides an isomorphism
\[X(A)/B^\perp\xrightarrow{\cong}X(\langle B\rangle).\]
We observe that
every nontrivial element $z\in\Sphere^1$ in its natural action on $\CC$ has $0\in\CC$ as its unique fixed point.
If $Y$ is a subgroup of $X(A)$ and if $a\in A$,
then 
\begin{equation}
 \label{summation}
 \sum_{\alpha\in Y}\alpha(a)=\begin{cases} |Y| &\text{ if }Y\subseteq a^\perp\\ 0 & \text{ else,}\end{cases}
\end{equation}
because the left-hand side is a fixed point for the action of the group $Y$ on $\CC$
through the character~$a$.
In particular, we have for $\alpha,\beta\in X(A)$ 
\[
 \langle \alpha, \beta\rangle=\sum_{a\in A}(\alpha\beta)(a)=
 \begin{cases} |A|&\text{ if }\alpha\beta=\eps\\0&\text{ else.}\end{cases}
\]
The Peter-Weyl Theorem \cite[Thm.~3.7]{HM} asserts among other things that the characters, viewed as complex functions on $A$,
span the complex vector space $R(A,\CC)=\CC^A$ and that they form an orthonormal basis with respect to the hermitian
inner product 
\[
 (\phi|\psi)=\frac{1}{|A|}\sum_{a\in A}\bar\phi(a)\psi(a).
\]
\end{Facts}
We apply this to finite fields.
Let $\FF_q$ denote the finite field of order $q$, for some prime power $q$. A \emph{multiplicative character} is a character for the cyclic group $\FF_q^\times$.
We extend the multiplicative characters to the monoid $(\FF_q,\cdot)$ by putting $\alpha(0)=0$.
\footnote{Most textbooks on elementary number theory 
define $\eps(0)=1$ for the trivial character $\eps$. However, this convention kills the group
structure on $X(\FF_q^\times)$, so we will not use it.}
\begin{Lem}\label{Convoluted}
Every real solution $\phi$ on the additive group $(\FF_q,+)$ to (\ref{Hopf}) can be written
as a complex linear combination of nontrivial multiplicative characters
\[
 \phi=\sum_{\alpha\neq \eps}c_\alpha\alpha,
\]
for complex coefficients $c_\alpha$, with $\bar c_\alpha=c_{\bar\alpha}$.
Conversely, any such linear combination satisfies the first two equations in (\ref{Hopf}).
\end{Lem}
\begin{proof}
We observed above that by the Peter-Weyl Theorem
every complex-valued function $\phi$ on $\FF_q$ which vanishes at $0$ can be written as a complex
linear combination of multiplicative characters, $\phi=\sum_{\alpha\in X(\FF_q^\times)}c_\alpha\alpha$.
For the constant function $1$ on $\FF_q$ we have
\[\langle 1,\phi\rangle=(q-1)(\eps|\phi)=(q-1)c_\eps,\]
hence $\langle1,\phi\rangle=0$ holds if and only if $c_\eps=0$.
The function $\phi$ is real if and only if $\phi=\bar\phi$.
\end{proof}
Now we need to evaluate convolutions of multiplicative characters $\alpha,\beta$.
For $\alpha,\beta\neq\eps$ we have by (\ref{summation})
\[
 (\alpha*\beta)(0)=\sum_{x\in\FF_q}\alpha(x)\beta(-x)=\beta(-1)\sum_{x\neq0}(\alpha\beta)(x)
 =\begin{cases} \beta(-1)(q-1)& \text{ if }\alpha\beta=\eps\\ 0 & \text{ else.}\end{cases}
\]
If $a\in\FF_q^\times$, then 
\begin{align*}
 (\alpha *\beta)(a)&=\sum_{x\in\FF_q}\alpha(x)\beta(a-x)\\
 &=\sum_{t\in\FF_q}\alpha(at)\beta(a-at)\\
 &=\sum_{t\neq0,1}\alpha(t)\beta(1-t)\alpha(a)\beta(a)\\
 &=J'(\alpha,\beta)(\alpha\beta)(a),
\end{align*}
where
\[
 J'(\alpha,\beta)=\sum_{t\neq0,1}\alpha(t)\beta(1-t)
\]
is a modified \emph{Jacobi sum}.
\footnote{Our $J'$ agrees with the classical Jacobi sum $J$ if $\alpha,\beta\neq\eps$, cp.~\cite{BEW} \S2.5.}
We note that $J'(\alpha,\beta)=J'(\beta,\alpha)$.
Properties of Jacobi sums are a classical topic in number theory, see eg.~\cite{IR} Ch.~8 \S3 and Ch.~10 \S3
or \cite{BEW} Ch.~2.
\begin{Lem}\label{Jacobi1}
Let $\alpha$ be a nontrivial character in $X(\FF^\times_q)$.
Then
\[
 J'(\eps,\eps)=q-2,\quad J'(\eps,\alpha)=-1\quad\text{ and }\quad J'(\bar\alpha,\alpha)=-\alpha(-1).
\]
If $\bar\alpha\neq\alpha$, then
\[
 |J'(\alpha,\alpha)|^2=q.
\]
\end{Lem}
\begin{proof}
The first two equations are immediate from the definition and (\ref{summation}).
If $\alpha\neq\eps$, then we have by (\ref{summation})
\[
 J'(\bar\alpha,\alpha)=\sum_{t\neq0,1}\alpha((1-t)/t)=\sum_{y\neq0,-1}\alpha(y)=-\alpha(-1).
\]
If $\alpha\neq\bar\alpha$, then 
\begin{align*}
 J'(\alpha,\alpha)J'(\bar\alpha,\bar\alpha) & =
 \sum_{x,y\neq0,1}\alpha(x)\alpha(1-x)\bar\alpha(y)\bar\alpha(1-y) \\
 &=\sum_{x,y\neq0,1}\alpha({\textstyle\frac{x}{y}})\alpha({\textstyle\frac{1-x}{1-y}})\\
 &=\sum_{x,y,x/y\neq0,1}\alpha({\textstyle\frac{x}{y}})\alpha({\textstyle\frac{1-x}{1-y}})+(q-2)\\
 &=\sum_{u,v,u/v\neq0,1}\alpha(u)\alpha(v)+(q-2)\\
 &=\sum_{u,v\neq0}\alpha(u)\alpha(v)-\sum_{u\neq0}\alpha^2(u)-2\sum_{u\neq0,1}\alpha(u)+(q-2)\\
 &=0-0-2(-1)+(q-2)\\
 &=q.
\end{align*}
where we made again use of (\ref{summation}). Now we observe that $\overline{J'(\alpha,\beta)}=J'(\bar\alpha,\bar\beta)$.
\end{proof}

\begin{Ex}[Quartic characters] 
Suppose that $q\equiv1\pmod 4$ is a prime power.
Let $\alpha\in X(\FF_q^\times)$ denote a multiplicative character of order $4$.
Then \[\phi=\alpha^2\] solves (\ref{Hopf}) on the abelian group $(\FF_q,+)$.
Indeed, $\alpha^2=\bar\alpha^2$, $\alpha^2(0)=0$,
$\alpha^2(-1)=1$ and for $a\neq0$ we have 
\[
 (\phi^2+\phi*\phi)(a)=(1+J'(\alpha^2,\alpha^2))\alpha^4(a).
\]
But $J'(\alpha^2,\alpha^2)=J'(\bar\alpha^2,\alpha^2)=-\alpha^2(-1)=-1$
and hence $(\phi^2+\phi*\phi)(a)=0$. This solution to (\ref{Hopf}), with 
\[
 n=q,\quad\theta=\hat\theta=0
\]
corresponds to the Paley graph, as described in Example~\ref{Palyegraph}.
\end{Ex}
\begin{Ex}[Octic characters] 
Suppose that $q\equiv1\pmod 8$ is a prime power.
Let $\alpha\in X(\FF_q^\times)$ denote a multiplicative character of order $8$,
let $c$ be a complex number of modulus $|c|=1$ and consider the function
\[\phi=c\alpha^2+\bar c\bar\alpha^2.\]
Then $\bar\phi=\phi$ and $\phi(0)=0$.
Moreover, $\alpha^2(-1)=1$ and thus $\phi(-x)=\phi(x)$.
We note that $\alpha^4=\bar\alpha^4$ is real.
For $a\neq0$ we have
\begin{align*}
 (\phi^2+\phi*\phi)(a) & =
 c^2(1+J'(\alpha^2,\alpha^2))\alpha^4(a)+\bar c^2(1+J'(\bar\alpha^2,\bar\alpha^2))\bar\alpha^4(a)
 +2(1+J(\alpha^2,\bar\alpha^2)) \\
 & = c^2(1+J'(\alpha^2,\alpha^2))\bar\alpha^4(a)+\bar c^2(1+J'(\bar\alpha^2,\bar\alpha^2))\alpha^4(a)\\
 & = 2\mathrm{Re}(c^2(1+J'(\alpha^2,\alpha^2)))\alpha^4(a).
\end{align*}
If we choose $c$ in such a way that $c^2(1+J'(\alpha^2,\alpha^2))$
has real part $0$, then $(\phi^2+\phi*\phi)(a)=0$.
Hence we obtain a solution with 
\[
n=q,\quad
 \theta=\hat\theta=0.
\]
If $q=p^{2s}$ for a prime $p\equiv3\pmod4$, this construction yields strongly regular graphs which were discovered by Peisert~\cite{Peisert}.
These graphs differ in general from the Paley graphs, although they have the same parameters $(n,r,\lambda,\mu)$.
Our general construction does not necessarily correspond to strongly regular graphs.
For example, if $q=73$, then $J(\alpha^2,\alpha^2)=3+8i$ and $\phi$ assumes four different values on $\FF_{73}^\times$.
\end{Ex}
\begin{Ex}[Cubic characters]\label{Cubic}
Suppose that $q\equiv1\pmod 3$ is a prime power.
Let $\alpha\in X(\FF_q^\times)$ denote a multiplicative character of order $3$,
let $c$ be a complex number of modulus $|c|=1$ and consider the function
\[
 \phi=c\alpha+\bar c\bar\alpha.
\]
We note that $\alpha(-1)=1$ since $\alpha$ has order~$3$.
Thus $\phi$ is real, $\phi(0)=0$, and $\phi(-x)=\phi(x)$.
For $a\neq0$ we obtain
\begin{align*}
 (\phi^2+\phi*\phi)(a) & =
 c^2(1+J'(\alpha,\alpha))\alpha^2(a)+\bar c^2(1+J'(\bar\alpha,\bar\alpha))\bar\alpha^2(a)
 +2(1+J(\alpha,\bar\alpha)) \\
 & = c^2(1+J'(\alpha,\alpha))\bar\alpha(a)+\bar c^2(1+J'(\bar\alpha,\bar\alpha))\alpha(a).
\end{align*}
and
\[
 (\phi^2+\phi*\phi)(0) = c^2(\alpha*\alpha)(0) + \bar c^2(\bar\alpha*\bar\alpha)(0)+2(\bar\alpha*\alpha)=2(q-1).
\]
For a solution to (\ref{Hopf}), we have necessarily $\theta=c^3(1+J'(\alpha,\alpha))\in\RR$
and this equation can be solved for $c$.
Then 
\[
 n=q,\quad\hat\theta=\frac{|1+J'(\alpha,\alpha)|}{\sqrt{2q(q-1)}},
\]
where $\frac{\sqrt{q}-1}{\sqrt{2q(q-1)}}\leq\hat\theta\leq \frac{\sqrt{q}+1}{\sqrt{2q(q-1)}}$.
%
\end{Ex}

\section{Concluding remarks}
We have seen that there are nontrivial solutions to the equations (\ref{Basic})
in all dimensions $n\geq 4$.
One series of solutions arises by Theorem~\ref{MainThm1} from strongly regular graphs on $n$ vertices.
These include Paley graphs, where $n=q\equiv1\pmod 4$ is a prime power, see Example~\ref{Palyegraph},
and disjoint unions of $m\geq 2$ complete graphs on $\ell\geq2$ vertices,
with $n=m\ell$, see Example~\ref{CompleteGraph}. From the Riemannian viewpoint, these correspond to products of round spheres.
The primitive strongly regular graphs, however, yield new and nontrivial solutions.

Another series of solutions arises from group rings of abelian groups $G$, with $n=|G|$.
If $G=L\times M$ is a product of two nontrivial abelian groups, then there are nonzero solutions, see Example~\ref{Composite}.
Also, there are solutions if $G$ is the additive group of a finite field
$\FF_q$, provided that $q\equiv1\pmod m$, where $m=3,4,8$, see Section~\ref{MultCharSection}.
Peter M\"uller has kindly pointed out that there are also solutions for $m=10$.
Numerical experiments with \textsf{SageMath} suggest the following conjecture.
\begin{Conj}
Let $q\geq 4$ be a prime power. Then 
the equations (\ref{Hopf}) in Theorem~\ref{MainThm2} have real solutions $\phi\neq0$,
where $G$ is the additive group of the finite field $\FF_q$ and $n=q$.
\end{Conj}

\subsection*{Acknowledgement}
We thank Peter M\"uller for interesting remarks on Section~\ref{MultCharSection}, and Christoph B\"ohm for 
many patient explanations about the $\#$-map and the Ricci flow.
Hans Cuypers and Ferdinand Ihringer suggested that certain association schemes might also lead to solutions.

\end{document}